\documentclass[twoside,11pt
]{amsart}

  \usepackage{a4wide}
  \usepackage[latin1]{inputenc}
  \usepackage{amsmath}
  \usepackage{amssymb}
  \usepackage{amsthm}
  \usepackage{mathrsfs}
  \usepackage{bbm}
  \usepackage{pifont}
  \usepackage{color}
  \usepackage{graphicx}
  \usepackage[normalem]{ulem}

	\newtheorem{theo}{Theorem}[section]

	\newtheorem{example}[theo]{Example}

	\newtheorem{lemma}[theo]{Lemma}
	
	\newtheorem{corollary}[theo]{Corollary}

\def\RR{{\ifmmode{\mathbbmss{R}}\else{$\mathbbmss{R}$}\fi}}

\def\NN{{\ifmmode{\mathbbmss{N}}\else{$\mathbbmss{N}$}\fi}}
\def\ZZ{{\ifmmode{\mathbbmss{Z}}\else{$\mathbbmss{Z}$}\fi}}

\title{Transposition game}
\author{\'Elise Janvresse, Steve Kalikow, Thierry de la Rue}
\address{\'Elise Janvresse, Thierry de la Rue:
Laboratoire de Math\'ematiques Rapha\"el Salem, 
Universit\'e de Rouen, CNRS -- 
Avenue de l'Universit\'e -- 
76801 Saint \'Etienne du Rouvray, France.}
\email{Elise.Janvresse@univ-rouen.fr\\Thierry.de-la-Rue@univ-rouen.fr}
\address{Steve Kalikow: University of Memphis, Department of Mathematics, Dunn Hall, 3725 Norriswood, Memphis, TN 38152, USA.}
\email{skalikow@memphis.edu}

\begin{document}
\maketitle

\bibliographystyle{plain}
\begin{abstract}
We introduce a two-player game, in which each player extends a given sequence by picking a free element in a domain $D\subset \RR$. 
The aim of the players is to control the parity of the number of transpositions necessary to put the final sequence in order.
We will see that the winner can be the last player, the second last player, the first player, the second player, the person who wants the parity to end up even or the person who wants the parity to end up odd. A special case of the game can be reduced to a game with nontrivial winning strategy, but describable in so simple a way that children can understand it and enjoy playing it. 
\end{abstract}

\section{Introduction}
Combinatorial games have aroused considerable interest, and gave rise to a vast mathematical literature (see e.g.~\cite{beck, berlekamp, conway}). Our purpose here is to introduce a new game, involving two elementary mathematical notions (ordered set, parity of a permutation), for which we give a complete analysis.

\subsection*{Description of the game}
Let $D$ be a finite or infinite subset of the real numbers, called the domain. 
A finite sequence of $m$ distinct real numbers in $D$ is given, as well as a positive integer $n$. 
We assume the number of elements of the domain $D$ is at least $m+n$. 

We play a game:
I go first. I extend the sequence by one term by adding an element of the domain that is not already in the sequence. Then you extend the sequence again by adding an element of the domain that is not already in the sequence. Then I go, then you go etc. until the sequence is extended to $m+n$ terms. 
If the number of transpositions necessary to put the final sequence in order is even then I win (or my object could be to make the number of transpositions necessary to be odd.) The purpose of this paper is to tell in every case who wins.

In most situations, the winning strategy relies on the central notion of pivot (see below). However, the winner strongly depends on the size and the structure of the domain $D$. We will see that all situations are possible: The winner can be the last player, the second last player, the first player, the second player, the person who wants the parity to end up even or the person who wants the parity to end up odd. 

The special case where the domain is finite of size $m+n+1$ requires a quite different analysis, and gives rise to a nontrivial winning strategy. This particular case will be coded with numbers, but can actually be played by a child who has virtually no concept of a number if he has simple objects available such as toothpicks or pennies. In the end, an interpretation of this case will be exhibited. 

We proceed the analysis of the game backwards, therefore it is convenient to call \textit{Black} the player who plays the last turn of the game and \textit{White} the player who plays the second last turn of the game. Therefore if $n$ is odd, Black is also the first player, whereas if $n$ is even, the first player is White.

\subsection*{$S$-blocks and pivots}

Let $S$ be the set of all numbers already in the sequence (which grows at each turn).
We introduce the notion of $S$-blocks by considering the equivalence relation $\sim_S$ given by:
$$
\mbox{For any } a, b\in S,   a \sim_S b \Longleftrightarrow [a, b]\cap D \subset S.
$$
We call \textit{$S$-block} the equivalence class for $\sim_S$.
An element in $D\setminus S$ is called a \textit{free} element.
A \textit{pivot} is an $S$-block containing an odd number of elements such that there exist two free elements, one of which is smaller and the other one is larger than the elements of the $S$-block.

In the paper, the parity of the number of transpositions necessary to put a finite sequence of distinct real numbers in order is called the \textit{parity of the sequence} (this is of course defined modulo~2).

\subsection*{Defining theorem and corollaries}
The following theorem is well known, but we present it anyway since it leads to useful corollaries for our analysis.

\begin{theo}\label{bla1}
Let $a_1,a_2,\dots, a_n$ be a sequence of distinct real numbers. The parity of the sequence is even (respectively odd) if 
$\sum_{i=1}^n |\{j<i: a_j>a_i\}|$ is equal to 0 (respectively 1) modulo 2.
\end{theo}
\begin{proof}
Observe that adding a new element $a_{m+1}$ to the sequence $(a_i)_{1\le i\le m}$ changes the parity of the number of transpositions necessary to put the new sequence in order if and only if the number $\sum_{i=1}^m 1_{\{a_i >a_{m+1} \}}$ of elements larger than $a_{m+1}$ is odd. 
\end{proof}

The \textit{end $S$-block} (respectively \textit{beginning $S$-block}) is the $S$-block containing the last (respectively first) term of the domain $D$. If there is no such block, which includes the possibility that there is no last (respectively first) term of $D$, then we say that the end $S$-block (respectively beginning $S$-block) is empty and hence that the number of terms in that block is 0.

\begin{corollary}\label{bla2}
If there is no pivot, then whoever plays next adds the number of elements of the end $S$-block to the parity of the sequence, no matter where he plays.
\end{corollary}

Because of Corollary~\ref{bla2}, it is of interest to know what the parity is about to be and what the parity plus the number of elements of the end $S$-block is about to be.

\begin{corollary}\label{bla4}
If there is a pivot, the next player can control both what the parity is about to be and what the parity plus the number of terms in the end $S$-block is about to be. Moreover, if there exist both a smallest free element larger than the pivot, and a largest free element smaller than the pivot, this can be done while removing the pivot.
\end{corollary}


\section{Finite case}
\label{Finite}
If the domain $D$ has size $d$, then it is isomorphic to $\{1, \dots , d\}\subset \ZZ_+$, so we can assume that $D=\{1, \dots , d\}$. 
Recall that, before the game begins, the set $S$ is of cardinality $m$ and the number $n$ of turns is such that $d\ge m+n$. The case $d=m+n+1$ is the special case where the game can be reduced to a children's game (see the annex).

In this setting, an $S$-block is a set of consecutive integers $\{\ell, \dots , \ell+j-1\}$ in $S$ such that both $\ell-1$ and $\ell+j$ do not belong to $S$. 
If $j$ is odd (respectively even), the $S$-block is said to be odd (respectively even).

\subsection{Case $d\not=m+n+1$}

\begin{theo}\label{FiniteCase}
If $d=m+n$, White wins. 

If $d>m+n+1$ and, in the initial sequence
	\begin{itemize}
	\item there is no pivot, White wins.
	\item there is exactly one pivot, the first player wins.
	\item there is more than one pivot, Black wins.
	\end{itemize}
\end{theo}

\begin{proof}[Proof of Theorem~\ref{FiniteCase} ($d=m+n$)]
Let us consider White's last turn: There are exactly two elements of the domain that are not already in the sequence.
By choosing one or the other, White can control the parity of the number of transpositions necessary to put the sequence in order. Therefore, White always wins.
\end{proof}

In order to prove the second part of Theorem~\ref{FiniteCase}, we need the following lemma.
\begin{lemma}
Assume $d>m+n+1$ and consider Black's second last turn. If there is no pivot, White wins. Otherwise, Black wins.
\end{lemma}

\begin{proof}
Suppose there is no pivot. 
If Black creates one, he loses by Corollary~\ref{bla4}. 
If Black does not create a pivot, he loses anyway: White attaches either to the beginning or end $S$-block, thereby controlling the parity plus the number of elements of the end $S$-block. This enable White to control the parity after Black plays by Corollary~\ref{bla2}.

If there is a pivot, then Black can create another pivot (note that at Black's second last turn, there are at least five elements of the domain that are not already in the sequence). 
Since White cannot remove them both, Black can control the parity at the end of the game by Corollary~\ref{bla4}. 
%
%
%
\end{proof}

Eventually, we see from this analysis that the aim of White is to have no pivot at Black's second last turn, whereas the aim of Black is to have at least one. We are now ready to prove the second part of Theorem~\ref{FiniteCase}.

\begin{proof}[Proof of Theorem~\ref{FiniteCase} ($d>m+n+1$)]
Assume there is no pivot in the initial sequence. White can always arrange such that there is no pivot before Black's turn:
White removes the pivot whenever Black creates one, or attaches to the beginning or end $S$-block. 
Thus White wins by Corollaries~\ref{bla2} and~\ref{bla4}.

Assume there are more than two pivots: Black can always create a new pivot and White cannot remove all of them, so Black wins. 

Assume now there is exactly one pivot, and let us describe the players' strategy.
If $n$ is even, the first player is White, who can remove the pivot. We are left with a situation without pivot, so White wins.
On the other hand, if $n$ is odd, the first player is Black, who can always create a new pivot. We are left with a situation with two pivots, so Black wins.
Therefore, in any case, the first player wins.
\end{proof}

When the size of the domain is strictly larger than $m+n+1$, the general rule of thumb in presence of pivots is that Black wins. The reason for this is that on Black's turn he can always create another pivot and if on his last turn there is a pivot remaining, he wins. White is kept busy killing pivots and cannot take time to alter a specific configuration that he might want to alter. 
This strategy can also be applied in the infinite case (Sections~\ref{infiniteZ} and~\ref{infiniteR}).

\subsection{Case $d=m+n+1$}
\label{Sec:finite_d=m+n+1}
What distinguishes the case $d=m+n+1$ from the case $d>m+n+1$ is that Black cannot always create another pivot at his second last turn because there is not always room to do it.
We now give a seemingly different game and prove it to be equivalent to our initial game. 
One of the interests of the new game is that no knowledge in mathematics is needed to play and enjoy it. We give a practical presentation to play it with simple objects in the annex.

Let $(n_1, \dots, n_{k})$ be a finite sequence of positive integers such that $\sum_{i=1}^{k}n_i\ge 2$. 
At each turn, this ordered sequence is replaced by a new one in the following way: 
The player can 
\begin{itemize}
\item[(i)] Replace $n_1$ by $n_1-1$ if $n_1\ge 2$;
\item[(ii)] Replace $n_{k}$ by $n_{k}-1$ if $n_{k}\ge 2$;
\item[(iii)] Remove any $n_i$ whenever $n_i=1$;
\item[(iv)] Choose an integer $n_i\ge 3$ and replace it by any pair $(n_i', n_{i+1}')$ with $n_i'+ n_{i+1}'=n_i-1$;
\item[(v)] Replace any $(n_i, n_{i+1})$ such that $n_i+ n_{i+1}\ge3$ by $n_i+ n_{i+1}-1$.
\end{itemize}
Observe that $\sum_{i=1}^{k}n_i$ decreases by $1$ at each turn.
The game ends when $\sum_{i=1}^{k} n_i = 2$: 
White is the winner if the final sequence is $(2)$, Black is the winner if the final sequence is $(1, 1)$.

\medskip
We interpret the ordered sequences as the number of free elements (that is elements of $D\setminus S$) between pivots.
More precisely, let $k\ge1$ be such that the number of pivots before the beginning of the game is $k-1$ and let us define the initial sequence $(n_i)_{1\le i\le k}$ of positive integers by
$$
n_1:= \begin{cases}
\Big|\left\{\ell\in D\setminus S: \ell \mbox{ is smaller than the elements of the first pivot}\right\} \Big| & \mbox{ if there is a pivot}\\
|D\setminus S| & \mbox{ otherwise}.
\end{cases}
$$
For $k-1\ge i\ge 2$, let $n_i$ be the number of free elements between the $(i-1)$-th odd $S$-block and the $i$-th odd $S$-block, and let 
$$
n_{k}:= 
\Big|\left\{\ell\in D\setminus S: \ell \mbox{ is larger than the elements of the last pivot} \right\} \Big|.
$$
Observe that $\sum_{i=1}^{k}n_i=d-m=n+1\ge 2$. Moreover, the parity of $\sum_{i=1}^{k}n_i$ tells us whose turn it is: $\sum_{i=1}^{k}n_i$ is even if and only if it is Black's turn.

With this interpretation of the ordered sequences, we are in the situations~(i) and~(ii) when the player respectively plays the smallest or the largest free element.
Situation~(iii) corresponds to choosing the (unique) free element between two consecutive odd $S$-blocks, which results in putting the two $S$-blocks together to get a new odd $S$-block.
Situation~(iv) corresponds to creating a new pivot.
At last, we are in situation~(v) when a player plays at the end of an odd $S$-block, which makes it even.

If the final sequence is $(1, 1)$, we are left with one pivot and it is Black's last turn. Thus Black can choose the parity.
If the final sequence is $(2)$, this means there are no pivot, thus Black cannot change the parity anymore. But at White's last turn, the sequence was either $(3)$ or $(2, 1)$ or $(1, 2)$. In the first case, White could control the parity by attaching to the beginning or end $S$-blocks, and in both other cases, White could control the parity by playing at one of the ends of the (unique) pivot. Thus, White is the winner.

Eventually, this proves that the two games are equivalent. We will now analyze our new game to describe each player's strategy and give the condition for Black to be the winner.

\medskip
Let $(n_1, \dots, n_{k})$ be the initial finite sequence of positive integers such that $\sum_{i=1}^{k}n_i\ge 2$. 
Here is the algorithm to decide which player has a winning strategy.
Color the odd $n_i$'s either green or red in the following way: Start with a green pen and change color each time you meet an even $n_i$. 
We define $\Delta(n_1, \dots, n_{k})$ as the difference between the number of green odd $n_i$'s and the number of red odd $n_i$'s. More formally, 
\begin{eqnarray*}
\Delta(n_1, \dots, n_{k}) &:=& 
\left| \left\{ i : n_i \mbox{ is odd and } \left(\sum_{j\le i} 1_{\{n_j\mbox{  is even}\}}\right) \mbox{ is even}\right\}\right| \\
&&- \left| \left\{ i : n_i \mbox{ is odd and } \left(\sum_{j\le i} 1_{\{n_j\mbox{  is even}\}}\right) \mbox{ is odd}\right\}\right| .
\end{eqnarray*}

Example: $\Delta(3,5,2,1,7,1,4,2,1)=-2$.

Observe that $\Delta(n_1, \dots, n_{k})$ has the same parity as $\sum_1^k n_i$, therefore at White's turn $\Delta$ is odd.

\begin{theo}\label{FiniteCase2}
At White's turn, White wins if and only if $|\Delta(n_1, \dots, n_{k})|= 1$. At Black's turn, Black wins if and only if $|\Delta(n_1, \dots, n_{k})|\neq 0$.
\end{theo}

\begin{proof}
We let the reader check that, at each turn, $|\Delta|$ increases or decreases by $1$. 

It should be noted that if all $n_i$'s are equal to $1$, then $|\Delta(n_1, \dots, n_{k})|= k$ and the only thing a player can do is decrease the number of ones by one. So the theorem follows when $n_i=1$ for all $i$.
Otherwise, we are going to show that, if $\Delta\neq0$, the player can always determine the way $|\Delta|$ evolves by choosing between two options as follows. We first consider the situation where there exists $n_i\ge3$:
\begin{itemize}
 \item If $n_i> 3$ and $n_i$ is odd: Replace $n_i$ by $(n_j', n_{j+1}')$, (with $n_i=n_j'+ n_{j+1}'+1$), where both numbers are odd or both are even.
 \item If $n_i> 3$ and $n_i$ is even: Replace $n_i$ by $(n_j', n_{j+1}')$, (with $n_i=n_j'+ n_{j+1}'+1$), where the first number is odd and the second is even or the first is even and the second is odd.
 \item If $n_i=3$ and $i=1$ or $i=k$: Replace it by $(1,1)$ or by $2$.
 \item If $n_i=3$ and $i\not=1, k$: Replace it by $(1,1)$ or replace $(n_i, n_{i+1})$ by $2+n_{i+1}$.
\end{itemize}
It remains to analyze the situations where there are only 1's and at least one 2. 
\begin{itemize}
 \item If $n_1=2$: Replace $n_1$ by $n_1-1=1$ ($\Delta$ is replaced by $1-\Delta$), or replace $(n_1,n_2)$ by $n_1+n_2-1=n_2+1$, ($\Delta$ is replaced by $\Delta+1$). 
\item If $n_k=2$:  Replace $n_k$ by $n_k-1=1$ or replace $(n_{k-1},n_k)$ by $n_{k-1},n_k-1=n_{k-1}+1$. 
\item If there exists $(n_{i},n_{i+1})=(1,2)$ with $1\le i\le k-2$: Remove $n_{i}=1$ or replace $(n_{i+1},n_{i+2})$ by $n_{i+1}+n_{i+2}-1=n_{i+2}+1$.
\end{itemize}

Observe that $\Delta(2)=0$ and $\Delta(1, 1)=2$, and recall that White wins if and only if the final sequence is $(2)$.
Therefore, Black's strategy is to increase $|\Delta|$, which is always possible except if there are only 1's (but then Black wins). On the other hand, White's strategy consists in decreasing $|\Delta|$.
\end{proof}


\section{Infinite case: $D\subset\ZZ$}\label{infiniteZ}

If $D\subset\ZZ$ is infinite, we can assume without loss of generality that $D=\ZZ$, $D=\ZZ_+$ or $D=\ZZ_-$ because every subset of $\ZZ$ is order isomorphic to one of those three.

\begin{theo}\label{InfiniteCaseZ}
Assume $D\subset\ZZ$ is infinite.
\begin{itemize}
\item If there is exactly one pivot, then the first player wins.
\item If there is more than one pivot, then Black wins.
\item If there is no pivot and 
	\begin{list}{$\circ$}{}
	\item $D=\ZZ$, then the second player wins.
	\item $D=\ZZ_+$, then the player who wants the parity of the final sequence to be the initial parity wins.
	\item $D=\ZZ_-$, then the parity of the final sequence is determined as a function of $n$, the initial parity and the number $\ell$ of terms in the end $S$-block: 
	If $n$ is even, the final parity is the initial one plus $n/2$ (modulo 2). 
	If $n$ is odd, the final parity is the initial one plus $\ell +(n-1)/2$ (modulo 2). 
	\end{list}
\end{itemize}
\end{theo}
Observe that the only difference with the finite case when $d>m+n+1$ corresponds to the situation without pivot.
\begin{proof}
If there is a pivot on Black's turn or more than one pivot on White's turn in any infinite case, Black wins because he keeps making more pivots and White cannot kill them all.

If there is exactly one pivot on White's turn, White can remove the pivot and control the parity by Corollary~\ref{bla4}. Then White removes the pivot whenever Black creates one and always control the parity. 
If Black does not create a pivot, this means he attaches to the beginning or end $S$-block (which is possible only if $D=\ZZ_+$ or $D=\ZZ_-$) and White also attaches to this $S$-block. This possibly modifies the parity when $D=\ZZ_-$ (Corollary~\ref{bla2}). But since White knows the number of left turns and controls the parity, he wins. 

Assume now there is no pivot. 
If $D=\ZZ$, then the first player has to form a pivot and the other party can kill it while controlling the parity. Hence the second player wins. 
If $D=\ZZ_+$ or $D=\ZZ_-$ then the person to play can avoid creating a pivot only by playing at the extremity. In fact no one will dare to create a pivot because if he does the other one kills it while controlling the parity and hence will win. 
If $D=\ZZ_+$, the parity never changes. If $D=\ZZ_-$, the parity is modified each time the end $S$-block is odd. 
Hence the theorem holds. 
\end{proof}

\section{Infinite case, $D\subset \RR$}\label{infiniteR}

We define the notion of $D$-blocks by considering the equivalence relations $\sim_D$ given by:
$$
 \mbox{For any }a, b\in D,   a \sim_D b \Longleftrightarrow [a, b]\cap D \mbox{ is finite.}
$$
A \textit{$D$-block} is an equivalence class for $\sim_D$.
Note that a $D$-block can be finite or isomorphic to $\ZZ$, $\ZZ_+$ or $\ZZ_-$.

We say that a $D$-block is \textit{exhausted} whenever it is also an $S$-block.
We say a block is \textit{central} if there exist two free elements, one of which is smaller and the other one is larger than the elements of the block.

Observe that a pivot is an odd central $S$-block.
Moreover, an odd central $D$-block which is exhausted is a pivot that cannot disappear. Indeed, for every central $D$-block there must be infinitely many free points between it and any other $D$-block, because otherwise it can be extended to a larger $D$-equivalence class. 

\medskip
We are about to make clear exactly who wins when it is Black's turn. The reader can easily deduce who wins when it is White's turn because White wins if and only if he can put Black in a losing position in his next turn.

We distinguish the case where all $D$-blocks are infinite (Theorem~\ref{Th:infinite}) from the one where there is at least one finite non-exhausted $D$-block (Theorem~\ref{Th:PresenceFinite}). We begin by proving the following lemma, which is useful in both cases.

\begin{lemma}\label{Lem:mM}
Let us consider Black's turn (but not his last turn).
Assume there exist a $D$-block with a largest free element $M\not=\max (D\setminus S)$ and a $D$-block with a smallest free element $m\not=\min (D\setminus S)$. Then Black wins.
\end{lemma}

\begin{proof}
If there is a pivot, Black can always create another one. Since White cannot remove more than one pivot at a time, there is a pivot at Black's last turn and Black wins, since he can choose the parity. 

Assume now there is no pivot. 
Observe that Black can always keep White busy until his second last turn by creating pivots, which White tries to remove. Therefore, we can assume we still are in this situation at Black's second last turn.

Assume $\{ m, M\}$ is disjoint from $\{\min (D\setminus S), \max (D\setminus S)\}$: By playing $m$ (respectively $M$), Blacks creates a pivot which White has to remove by playing immediately to the right  (respectively left) of the pivot (if this is not possible, then there is a pivot at Black's last turn, so Black wins). 
The parity of the permutation before Black's last turn is different according to whether Black chose to play $m$ or $M$ by Corollary~\ref{bla2}. Therefore, Black controls the final parity. 

If $m=\max (D\setminus S)$ and $M\not= \min (D\setminus S)$, one choice for Black is to play $m$, which forces White to play $\min (D\setminus S)$ (otherwise, since $D\setminus S\cup\{m\}$ has no maximum element, White creates a pivot). The other possible choice for Black is to play $M$; This creates a pivot, so White has to play immediately to the left. At this stage, the parity of the permutation does not depend on Black's choice but the parity will be different after Black's last turn, since the parity of the number of used elements in the end $S$-block is different (cf. Corollary~\ref{bla2}).

If $M=\min (D\setminus S)$ and $m\not= \max (D\setminus S)$, Black plays $M$, which forces White to create a pivot. Thus Black wins.

Assume now that $m=\max (D\setminus S)$ and $M= \min (D\setminus S)$: If Black plays $m$ (respectively $M$), then White is forced to play $M$ (respectively $m$), since any choice of a non-extremal element would create a pivot. But again, the two choices give rise to different parities at Black's last turn, and no matter where he plays, the difference of parity remains. Hence Black wins.
\end{proof}

\begin{theo}[Infinite $D$-blocks]
\label{Th:infinite}
Assume that all $D$-blocks are infinite and let us consider Black's turn (but not his last turn).
If there is a pivot, Black wins. Otherwise, we are in one of the following situations (see Figure~\ref{fig:infinite}): 
\begin{enumerate}
\item If there exist a $D$-block with a largest free element $M$ and a $D$-block with a smallest free element $m$ and 
	\begin{enumerate}
	\item one can choose $m\not=\min (D\setminus S)$ and $M\not=\max (D\setminus S)$: Black wins;
	\item $\max (D\setminus S)$ does not exist, and the only choice for $m$ is $\min (D\setminus S)$: Black wins; 
	\item $\min (D\setminus S)$ does not exist, and the only choice for $M$ is $\max (D\setminus S)$: Black wins; 
	\item $\min (D\setminus S)$ exists, one can choose $m\not=\min (D\setminus S)$, but the only choice for $M$ is $\max (D\setminus S)$: The player who wants to keep the initial parity wins;
	\item $\max (D\setminus S)$ exists, one can choose $M\not=\max (D\setminus S)$ but the only choice for $m$ is $\min (D\setminus S)$: The player who wants the final parity to be the initial one plus $\ell+(n-1)/2$ wins, where $\ell$ is the number of terms in the end $S$-block;
	\item The only choice for $(m, M)$ is $(\min (D\setminus S), \max (D\setminus S))$: White wins;
	\end{enumerate}
\item If there exist a $D$-block with a smallest free element $m$ and no $D$-block with a largest free element, then the player who wants to keep the initial parity wins;
\item If there exist a $D$-block with a largest free element $M$ and no $D$-block with a smallest free element, then the player who wants the final parity to be the initial one plus $\ell+(n-1)/2$ wins, where $\ell$ is the number of terms in the end $S$-block;
\item If none of the above is valid, then there exist no $D$-block with a smallest free element and no $D$-block with a largest free element: White wins.
\end{enumerate}
\end{theo}

\begin{proof}
If there is a pivot, Black can always create another one. Since White cannot remove more than one pivot at a time, there is a pivot at Black's last turn and Black wins. 

Assume now there is no pivot. 

Let us analyze Situations~(1)(f) and~(4): 
If Black creates a pivot, White can always remove it and choose the parity at the same time. 
The only way for Black not to create a pivot is to play $\min (D\setminus S)$ or $\max (D\setminus S)$ (if they exist). 
In this case, by Corollary~\ref{bla2} White has no control of the parity after he plays. However he can avoid creating a pivot either by playing the minimum free element or the largest free element. By choosing one of those two, he determines the parity of the number of terms in the end $S$-block. Again by Corollary~\ref{bla2}, that determines what the final parity will be after Black plays. Thus White wins.

Assume we are in Situations~(1)(a), (1)(b) or~(1)(c).
Black can always keep White busy until his second last turn by playing inside an infinite $D$-block, creating pivots which White has to remove. 
Therefore, we can assume we still are in one of these situations at Black's second last turn.

Situation~(1)(a) is a direct application of Lemma~\ref{Lem:mM}. 
In Situation~(1)(b) (respectively ~(1)(c)), Black does not create a pivot by playing $m$ (respectively $M$), but White is forced to play immediately to the right (respectively left) anyway: Indeed, since $\max (D\setminus S)$ (respectively $\min (D\setminus S)$) does not exist, White would create a pivot by choosing another element. The same argument applies and Black controls the final parity. 

Let us now turn to Situation~(2).
Black has to choose the leftmost free element of a $D$-block (otherwise, White can both remove the pivot created by Black and control the parity by Corollary~\ref{bla4}). Then White has to play immediately to the right (otherwise he would leave a pivot), but does not control the parity which is unchanged. 
In Situation~(1)(d), Black has to create a pivot (otherwise, White can choose the parity by playing $\min (D\setminus S)$ or $\max (D\setminus S)$). Since Black does not want White to control the parity when removing the pivot, he has to choose the leftmost free element of a central $D$-block (e.g. $m$). The same conclusion as in Situation~(2) follows.

Recall that $\ell$ denotes the number of terms in the end $S$-block. 
Situations~(3) and~(1)(e) are symmetrical to Situations~(2) and ~(1)(d) and the same arguments hold, but the parity is modified every second turn (at White's turn or Black's turn, depending on the parity of $\ell$).
\end{proof}

For short, let \textit{FNEDB} stand for ``finite non-exhausted $D$-block''.

\begin{theo}[Presence of a finite $D$-block]
\label{Th:PresenceFinite}
Assume there is at least one FNEDB and consider Black's turn (but not his last turn).
If there is a pivot, Black wins. Otherwise, we are in one of the following situations (see Figure~\ref{fig:finite}): 
\begin{itemize}
\item[(5)] There exists a central FNEDB: Black wins;
\item[(6)] There exist two noncentral FNEDB: Black wins;
\item[(7)] The rightmost $D$-block is the only FNEDB and 

(a) there is another $D$-block with a largest free element: Black wins;

(b) $\min (D\setminus S)$ does not exist: Black wins;

(c) none of the above, the player who wants the final parity to be the initial one wins.

\item[(8)] The leftmost $D$-block is the only FNEDB and 

(a) there is another $D$-block with a smallest free element: Black wins;

(b) $\max (D\setminus S)$ does not exist: Black wins;

(c) none of the above, the player who wants the final parity to be the initial one plus $\ell+(n-1)/2$ wins, where $\ell$ is the number of terms in the end $S$-block.
\end{itemize}
\end{theo}

\begin{proof}
Situations~(5) and~(6) are direct applications of Lemma~\ref{Lem:mM}. 
Let us turn to Situation~(7) which is more intricate. 
If there is another $D$-block with a largest free element, then Lemma~\ref{Lem:mM} applies : Black wins. 
Assume now that $\min (D\setminus S)$ does not exist. Black can create pivots that White has to remove, until it is Black's second last turn: Then, Black can choose the final parity by playing either the minimum or the maximum free element of the finite block (if there is only one point left in the block, Black plays it and White is forced to create a pivot).
Otherwise, $\min (D\setminus S)$ exists and the rightmost $D$-block is the only one with a largest free element: 
if the number of free elements in the rightmost $D$-block is at least $n$, then Black uses the same strategy as in Situation~(1)(d) of Theorem~\ref{Th:infinite}. 

Situation~(8) is symmetrical to Situation~(7), the parity being modified every second turn (at White's turn or Black's turn, depending on the parity of $\ell$).
\end{proof}

\section{Examples}

Here is a series of examples illustrating different cases, including situations where White begins.
In all these examples, assume $S$ to be empty and the number $n$ of left turns to be at least 3.

\bigskip

\begin{example}
 $D=\{1,\ldots,d\}$ and $3\le n\le d$: White wins.
\end{example}
If $d=n$ or $d>n+1$, Theorem~\ref{FiniteCase} applies. If $d=n+1$, we use Theorem~\ref{FiniteCase2}, with the sequence $(n_i)$ reduced to $(d)$. We have to distinguish the cases $d$ odd and $d$ even. If $d$ is odd, it is White's turn, and $\Delta(d)=1$. If $d$ is even, it is Black's turn and $\Delta(d)=0$.
In any case, White wins.

\bigskip

\begin{example}
 $D$ contains an interval with nonempty interior: Black wins.
\end{example}
Each point in the interior of the interval is an FNEDB, hence we are in situation~(5) of Theorem~\ref{Th:PresenceFinite}.

\bigskip

\begin{example}
 $D=\{2-1/k,\ k\ge1\}\cup\{3-1/k,\ k\ge1\}$: The person who wants the parity to be even wins.
\end{example}
The domain is constituted of 2 $D$-blocks which are isomorphic to $\ZZ_+$. If Black begins, we are in situation~(2) of Theorem~\ref{Th:infinite}, hence initial parity (which is even since $S$ is empty) wins. If White begins, since he cannot create a pivot, he has to play the minimum of the domain and we are left with a similar situation at Black's turn (with $S$ reduced to a singleton, the parity remains even).

\bigskip

\begin{example}
 $D=\{2+1/k,\ k\ge1\}\cup\{3+1/k,\ k\ge1\}$: The person who wants the parity to be $[n/2]$ wins.
\end{example}
The domain is constituted of 2 $D$-blocks which are isomorphic to $\ZZ_-$. If Black begins, $n$ is odd and we are in situation~(3) of Theorem~\ref{Th:infinite}: The winner is the player who wants the parity to be $(n-1)/2=[n/2]$ (there are no term in the end $S$-block). If White begins, since he cannot create a pivot, he has to play the maximum of the domain and we are left with a similar situation at Black's turn, with one element in the end $S$-block. By (3) of Theorem~\ref{Th:infinite}, the winner is the player who wants the parity to be $1+(n-2)/2=[n/2]$.

\bigskip

\begin{example}
 $D=\{2-1/k,\ k\ge1\}\cup\{2+1/k,\ k\ge1\}$: White wins.
\end{example}
The domain is constituted of 2 $D$-blocks which are isomorphic to $\ZZ_+$ and $\ZZ_-$. By~(1)(f) of Theorem~\ref{Th:infinite}, if Black begins then White wins. But if White begins, he can play the minimum of $D$ and leave Black with a losing situation.

\bigskip

\begin{example}
 $D=\{2+1/k,\ k\ge1\}\cup\{5-1/k,\ k\ge1\}$: Second player wins.
\end{example}
There is only one $D$-block which is isomorphic to $\ZZ$: Theorem~\ref{InfiniteCaseZ} applies.

\bigskip

\begin{example}
 $D=\{1\}\cup\{2+1/k,\ k\ge1\}\cup\{5-1/k,\ k\ge1\}$: First player wins.
\end{example}
First player plays 1, and we are left with the previous example.

\bigskip

\begin{example}
 $D=\{1,2\}\cup\{2+1/k,\ k\ge1\}\cup\{5-1/k,\ k\ge1\}$: Black wins.
\end{example}
There are 2 $D$-blocks, the first one containing 2 elements and the second one isomorphic to $\ZZ$. If Black begins, we can apply~(8)(b) of Theorem~\ref{Th:PresenceFinite}, and Black wins. If White begins, he is forced to play 1 in order not to create a pivot and we are left with the previous example.


\section*{Annex: Children's games}\label{annex}
In this section, we propose two simple two-player games equivalent to the transposition game in the finite case with $d=m+n+1$ (see Section~\ref{Sec:finite_d=m+n+1}).

\subsection*{Playing the game with pennies}
Start with a finite number of pennies, split into clumps disposed along a line. Black is the player when the number of pennies is even. 
At each turn, the player can 
\begin{itemize}
\item Remove a penny in the first or in the last clump;
\item Remove a clump reduced to one penny; 
\item Remove one penny from a clump containing more than 3 pennies and split the rest into two nonempty clumps;
\item Merge two adjacent clumps and remove one penny in the resulting clump.
\end{itemize}
The game ends when there are two pennies left. 
White is the winner if there is only one clump, Black is the winner if there are two clumps at the end.

\medskip
The above game is a direct interpretation of the transposition game in the finite case with $d=m+n+1$ (see Section~\ref{Sec:finite_d=m+n+1}): Let $n+1\ge 2$ be the number of pennies. 
For any $1\le i\le k$, let $n_i\ge 1$ be the number of pennies into the $i$-th clump : $\sum_{i=1}^{k}n_i = n+1 \ge 2$.
Each possible action of one player modifies the $(n_i)$'s as in Section~\ref{Sec:finite_d=m+n+1}.

\subsection*{Playing the game with black and white pieces}
Start with a finite number of black and white pieces, disposed along a line.  
In this setting, to be consistent with the rest of the paper, Black is the player when the number of pieces is odd. 
Each player may at its turn, independently of its own color, do one of the following:
\begin{itemize}
\item remove one black piece;
\item replace two consecutive white pieces with a black one;
\item remove one extreme piece (the leftmost or the rightmost, whatever its color).
\end{itemize}
Observe that at each step, the total number of pieces decreases by one. 
If the last piece is white, then White wins, else Black wins. An example of this game is shown on Figure~\ref{fig:pieces}.

\medskip
The above game is equivalent to the transposition game in the finite case with $d=m+n+1$ (see Section~\ref{Sec:finite_d=m+n+1}). 
Let $n$ be the number of pieces and let $k\in\{1, \dots, n+1\}$ be the number of black pieces plus one at the beginning of the game. 
Put a bar after each black piece.
Let $n_1\ge1$ be the number of pieces before the first bar. 
For $2\le i\le k-1$, let $n_i\ge1$ be the number of pieces between the $(i-1)$-th bar and the $i$-th bar. 
Let $n_k\ge1$ be the number of pieces after the last bar (the $k$-th bar) plus one.
Observe that $\sum_{i=1}^k n_i = n+1\ge 2$.

Let us analyze how the action of one player modifies the $(n_i)$'s: 
Removing the $i$-th black piece corresponds to removing $n_i$ (if $n_i=1$) or replacing $(n_i, n_{i+1})$ by $n_i+ n_{i+1}-1$ (options (iii) and (v) of Section~\ref{Sec:finite_d=m+n+1}); 
Replacing two consecutive white pieces with a black one corresponds to replacing some $n_i\ge 3$ by $(n_i', n_{i+1}')$ with $n_i'+ n_{i+1}'=n_i-1$ (option (iv) of Section~\ref{Sec:finite_d=m+n+1}); 
Removing one extreme white piece decrements $n_1$ or $n_{k}$ by one (options (i) and (ii) of Section~\ref{Sec:finite_d=m+n+1}); 
Removing one extreme black piece suppresses $n_1$ or $n_{k}$, which was equal to one (option (iii) of Section~\ref{Sec:finite_d=m+n+1}).
After $n-1$ steps, if the last piece is Black (respectively White), then $k=2$ and $(n_1, n_2)=(1, 1)$ (respectively $k=1$ and $n_1=2$) and Black (respectively White) wins.

\begin{figure}[p]
\input{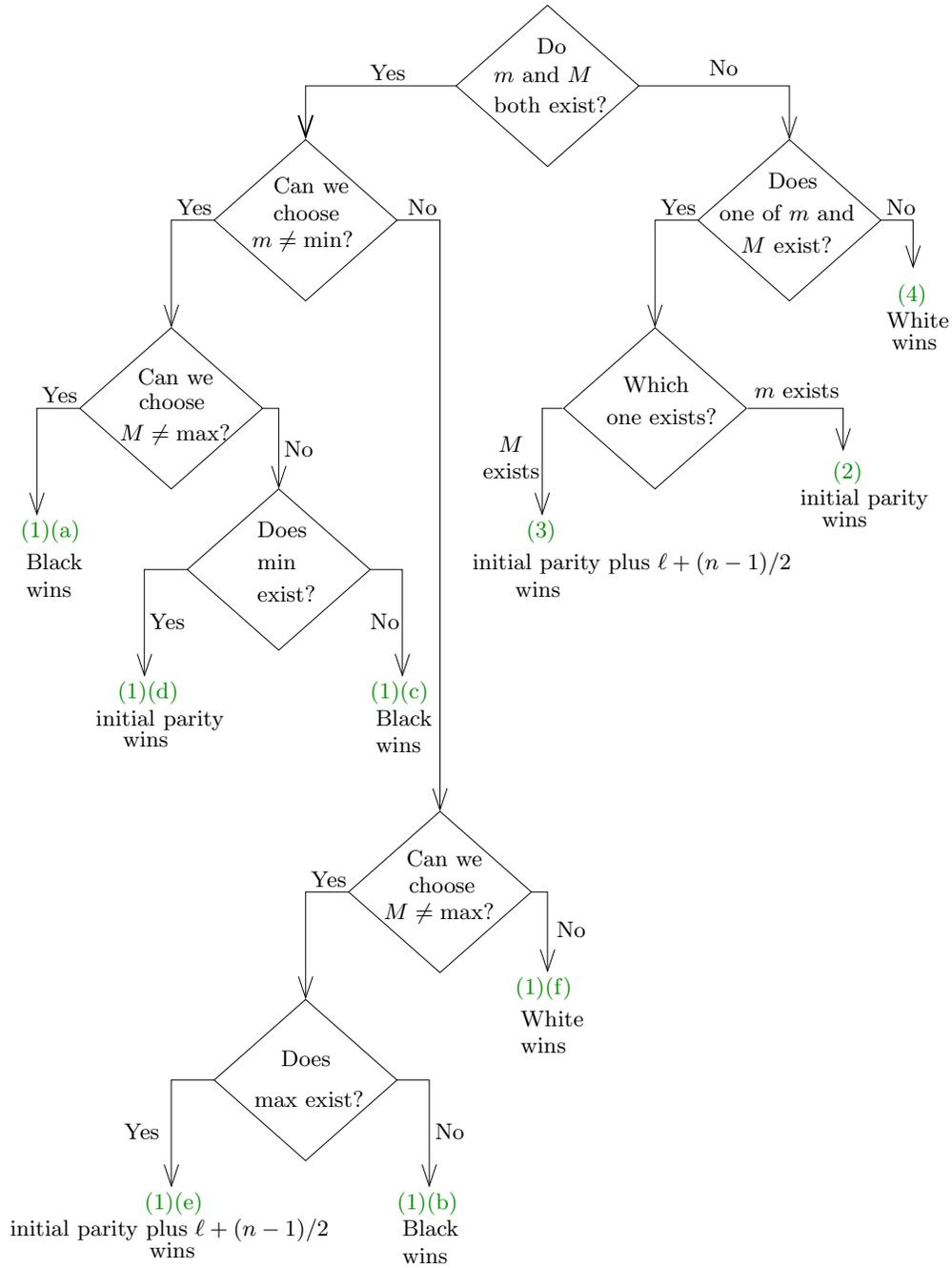}
\caption{(Infinite $D$-blocks.) Consider Black's turn (but not his last turn). Assume that all $D$-blocks are infinite and that there is no pivot. Whenever it is possible, $m$ (respectively $M$) is to be chosen among the smallest (respectively largest) free elements of $D$-blocks. Recall that we denote by $n$ the number of turns and by $\ell$ the number of terms in the end $S$-block. 
For short, $\min$ and $\max$ denote $\min (D\setminus S)$ and $\max (D\setminus S)$.}
 \label{fig:infinite}
\end{figure}

\begin{figure}[p]
\input{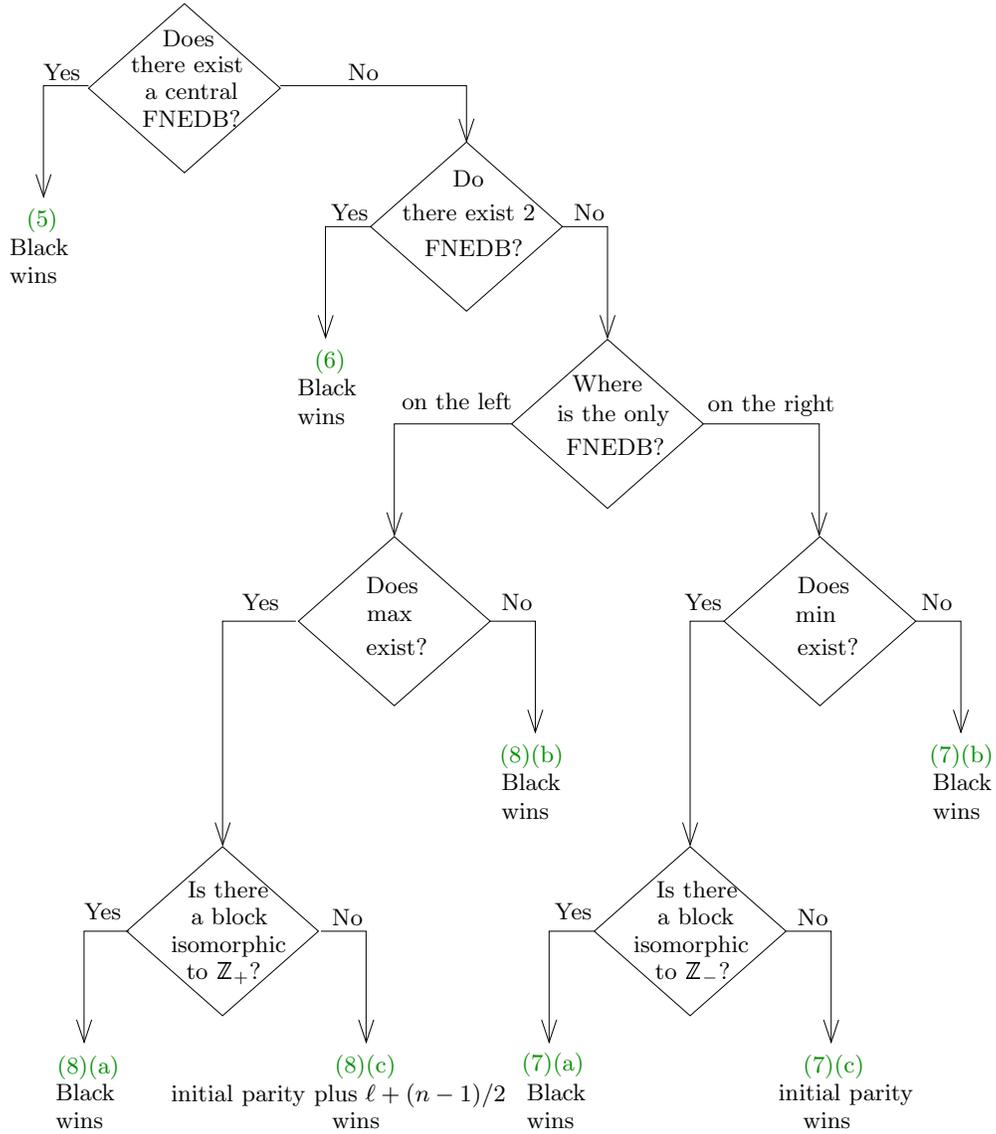}
\caption{(Presence of a finite $D$-block.) Consider Black's turn (but not his last turn). Assume there is at least one FNEDB and that there is no pivot. Notations are the same as in Figure~\ref{fig:infinite}.}
\label{fig:finite}
\end{figure}

\begin{figure}
\label{fig:pieces}
\input{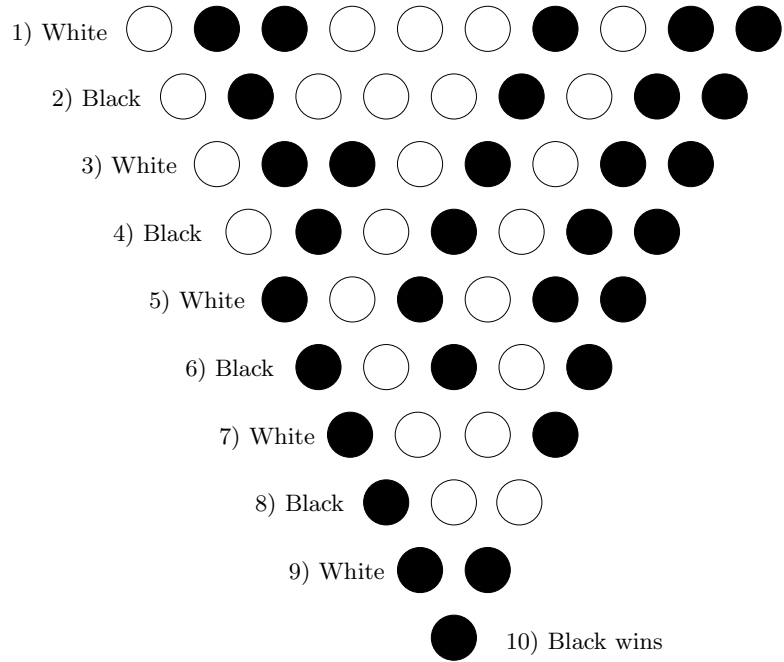}
\caption{Example of game with $n=10$ (first player is White).}
\end{figure}

\bibliography{game}

\begin{thebibliography}{1}

\bibitem{beck}
J{\'o}zsef Beck.
\newblock {\em Combinatorial games - Tic-tac-toe theory}, volume 114 of {\em
  Encyclopedia of Mathematics and its Applications}.
\newblock Cambridge University Press, Cambridge, 2008.

\bibitem{berlekamp}
Elwyn~R. Berlekamp, John~H. Conway, and Richard~K. Guy.
\newblock {\em Winning ways for your mathematical plays. {V}ol. 1-4}.
\newblock A K Peters Ltd., Natick, MA, second edition, 2003.

\bibitem{conway}
J.~H. Conway.
\newblock {\em On numbers and games}.
\newblock A K Peters Ltd., Natick, MA, second edition, 2001.

\end{thebibliography}

\end{document}